\documentclass{amsart}
\usepackage{amsfonts}
\usepackage{amsmath,amssymb}
\usepackage{amsthm}
\usepackage{amscd}
\usepackage{graphics}
\usepackage{graphicx}

\theoremstyle{remark}{
\newtheorem{Def}{{\rm Definition}}
\newtheorem{Ex}{{\rm Example}}

}
\theoremstyle{plain}
{

\newtheorem{Prop}{Proposition}
\newtheorem{Thm}{Theorem}

}

\begin{document}
\title[Regions surrounded by cylinders of real algebraic hypersurfaces]{Regions surrounded by cylinders of real algebraic manifolds and natural decompositions}
\author{Naoki kitazawa}
\keywords{(Non-singular) real algebraic manifolds and real algebraic maps. Real algebraic regions. Singularity theory of differentiable maps. \\
\indent {\it \textup{2020} Mathematics Subject Classification}: 05C05, 05C10, 14P05, 14P10, 14P25, 57R45, 58C05.}

\address{Institute of Mathematics for Industry, Kyushu University, 744 Motooka, Nishi-ku Fukuoka 819-0395, Japan\\
 TEL (Office): +81-92-802-4402 \\
 FAX (Office): +81-92-802-4405 \\
}
\email{naokikitazawa.formath@gmail.com}
\urladdr{https://naokikitazawa.github.io/NaokiKitazawa.html}
\maketitle
\begin{abstract}
The author has been interested in regions surrounded by cylinders of real algebraic hypersurfaces and their shapes and polynomials associated to them. Here, we formulate and investigate natural decompositions into such cylinders of real algebraic hypersurfaces. Especially, intersections of these cylinders of real algebraic hypersurfaces, which give important information on regions, are investigated via singularity theory.

This is a kind of natural problems on real geometry. This also comes from construction of explicit real algebraic maps onto explicit regions in real affine spaces on real algebraic manifolds. More generally, we are interested in difficulty in explicit construction of real algebraic objects, where existence and approximation has been well-known, since pioneering studies by Nash and Tognoli, in the latter half of 20th century. This also comes from interest in singularity theory of differentiable, smooth or real algebraic functions and maps, especially, explicit construction. 

\end{abstract}
\section{Introduction.}
\label{sec:1}
Regions in the $k$-dimensional Euclidean space ({\it real affine space}) ${\mathbb{R}}^k$ ($\mathbb{R}:={\mathbb{R}}^1$) surrounded by hypersurfaces are fundamental geometric objects. We focus on these regions surrounded by cylinders of real algebraic hypersurfaces and their shapes and polynomials associated to them. We are also mainly concerned with intersections of these real algebraic hypersurfaces as important information on the shapes.

As complex algebraic geometry, real algebraic geometry has a nice history. Studies on existence and approximation have been developing, since celebrated studies by \cite{nash, tognoli}. See also \cite{kollar, lellis}.
One of elementary natural, interesting, and recent study is \cite{bodinpopescupampusorea}. 
In this, Bodin, Popescu-Pampu and Sorea consider planar graphs embedded in a generic way and find regions naturally collapsing to the graphs by approximation respecting derivatives.
The author has started to consider that this is regarded as the image of a nice smooth map of a certain class containing the canonical projections of the unit spheres as simplest cases. These maps are so-called {\it special generic} maps, explained in \cite{saeki1}. For related studies on special generic maps, see also \cite{kitazawa15}, where cohomology rings of the manifolds are completely determined by the images, smoothly immersed manifold of codimension $0$. We construct such maps in the real algebraic situation in \cite{kitazawa3}, first, remarked additionally in \cite{kitazawa8} by the author.
After that, we have considered naturally extended situations and succeeded in related construction for example, in \cite{kitazawa4, kitazawa5, kitazawa6, kitazawa7, kitazawa9, kitazawa10, kitazawa11, kitazawa12, kitazawa13, kitazawa14, kitazawa15, kitazawa16, kitazawa18} for example.
For Definition \ref{def:1}, see \cite{kitazawa16, kitazawa18} mainly and in addition, \cite{kitazawa9, kitazawa10, kitazawa11} for example. In our paper, we also use tools and methods from some singularity theory of differentiable maps, presented in \cite{golubitskyguillemin}.

Let ${\rm id}_X:X \rightarrow X$ denote the identity map on $X$. 
For a subspace $Y \subset X$ in a topological space $X$, let ${\overline{Y}}^X$ denote its closure in $X$.
For a differential manifold $X$, let $T_p X$ denote the tangent vector space at $p \in X$.
A {\it singular} point $p$ of a differential map $c:X \rightarrow Y$ is a point, where the differential ${dc}_p:T_p X \rightarrow T_{c(p)} Y$, a linear map, is not of the full rank.
A {\it real algebraic} manifold is a union $Z_{c}$ of connected components of the zero set $c^{-1}(0)$ of a real polynomial map $c:{\mathbb{R}}^{k_1} \rightarrow {\mathbb{R}}^{k_2}$ such that $c {\mid}_{Z_{c}}$ has no singular point.
\begin{Def}
\label{def:1}
Let $D \subset {\mathbb{R}}^n$ be a non-empty, connected and open set of ${\mathbb{R}}^n$. Let $\{S_j\}_{j=1}^{l}$ be a family of real algebraic manifolds of dimension $n-1$ in ${\mathbb{R}}^n$ each $S_j$ of which is a union of some components of the zero set of a real polynomial $f_j$ and that $S_j \bigcap {\overline{D}}^{{\mathbb{R}}^n} \neq \emptyset$ for all integers $1 \leq j \leq l$. We also assume the following.
\begin{enumerate}
\item \label{def:1.1} There exists a small open neighborhood $U_D$ of the closure ${\overline{D}}^{{\mathbb{R}}^k}$ of $D$ with $D=U_D \bigcap {\bigcap}_{j=1}^{l} \{x \mid f_j(x)> 0\}$ and $U_D \bigcap S_j=U_D \bigcap \{f_j(x)=0\}$.
\item \label{def:1.2} For any increasing subsequence $\{i_j\}_{j=1}^{l^{\prime}}$ of the sequence $\{j\}_{j=1}^l$ of positive integers and any point  $p \in {\bigcap}_{j=1}^{l^{\prime}} S_{i_j} \bigcap {\overline{D}}^{{\mathbb{R}}^n}$,  $\dim {\bigcap}_{j=1}^{l^{\prime}} T_p S_{i_j}=n-l^{\prime}$ holds.
\end{enumerate}
We call the pair $(D,\{S_j\}_{j=1}^l)$ a {\it real algebraic region} or an {\it RA-region} of ${\mathbb{R}}^n$.  
\end{Def}
Hereafter, we also see $\mathbb{R}$ as the ordered set. Let $\mathbb{N} \subset \mathbb{R}$ denote all the set of positive integers. For an ordered set $P$ and an element $p$, let $P_{p}$ denote the set of all elements smaller than or equal to $p$. We also see ${\mathbb{R}}^n$ as the real vector space with $e_i$ denoting a vector the value of whose $i$-th component is $1$ and those of whose remaining components are $0$. 
Let $N_{n^{\prime}} \subset {\mathbb{N}}_n$ denote a subset consisting of exactly $n^{\prime}$ numbers with $1 \leq n^{\prime} \leq n$.
Let ${\pi}_{n,N_{n^{\prime}}}:{\mathbb{R}}^n \rightarrow {\mathbb{R}}^{n^{\prime}}$ be the projection defined by ${\pi}_{n,N_{n^{\prime}}}(x):=x^{\prime}$ with
$x=(x_1,\cdots x_n)=(x_j)_{j \in {\mathbb{N}}_n}$ and $x^{\prime}=(x_{j^{\prime}})_{j^{\prime} \in N_{n^{\prime}}}$.
\begin{Def}
\label{def:2}
A {\it cylinder of a real algebraic manifold} in ${\mathbb{R}}^n$ with $n \geq 2$ is a real algebraic manifold $C$ (of dimension at most $n-2$) represented as a union of connected components of a real polynomial function $f(x)$ whose values at points of $C$ depend only on at most $n-1$ variables of $\{x_j\}_{j \in {\mathbb{N}}_n}$ in $x=(x_1,\cdots x_n)=(x_j)_{j \in {\mathbb{N}}_n}$, labeled by a proper subset $A_C \subset {\mathbb{N}}_n$ as $\{x_j\}_{j \in A_C}$. 
\end{Def}

\begin{Def}
\label{def:3}
In Definition \ref{def:1}, we define the following.
\begin{enumerate}
\item We define the set $S_{i_j,1 \leq j \leq l^{\prime}}$ of all points of ${\overline{D}}^{{\mathbb{R}}^n-D}$ contained in $l^{\prime}<n$ distinct ($n-1$)-dimensional manifolds $S_{i_j}$ and not contained in any other one $S_j$, which is an ($n-l^{\prime}$)-dimensional manifold with no boundary. A point $p \in S_{i_j,1 \leq j \leq l^{\prime}}$ is called an {\it $N_1$-critical} point of $(D,\{S_j\}_{j=1}^l)$ if $p$ is a critical point of the restriction ${\pi}_{n,N_1} {\mid}_{S_{i_j,1 \leq j \leq l^{\prime}}}$.
\item A point $p$ in exact one ($n-1$)-dimensional manifold $S_{i_0}$ is called a {\it normal} point of $(D,\{S_j\}_{j=1}^l)$ and if it is not an $N_1$-critical point of $(D,\{S_j\}_{j=1}^l)$ for some one element set $N_1 \subset {\mathbb{N}}_n$, then it is called an {\it $N_1$-normal} point or an {\it $N_1$-N-point} of $(D,\{S_j\}_{j=1}^l)$.

\end{enumerate}

\end{Def}

It is a kind of fundamental, natural and important problems to determine whether a region surrounded by cylinders of real algebraic manifolds is an RA-region. The following show our new answer. This is also one of our main result.
\begin{Thm}
\label{thm:1}
Let $\{S_j\}_{j=1}^{l}$ be a family of real algebraic manifolds of dimension $n-1$ in ${\mathbb{R}}^n$ and $\{A_i\}_{i=1}^a$ a family of $a \in \mathbb{N}$ distinct subsets of ${\mathbb{N}}_n$ such that $A_{i_1} \subset A_{i_2}$ does not hold for any pair of distinct integers $1 \leq i_1, i_2 \leq n$ and that ${\mathbb{N}}_n={\bigcup}_{i=1}^a A_i$. 
\begin{enumerate}
	\item \label{thm:1.1} A subset $A_{S_j} \subset {\mathbb{N}}_n$ is defined for each $S_j$, for some $A_i$, $A_i=A_{S_j}$, and for each $A_i$ with $i \in {\mathbb{N}}_a$, $A_i=A_{S_{i_j}}$ for some $i_j$.
	\item \label{thm:1.2} Each $S_j$ of which is a cylinder of a real algebraic manifold and a union of some components of the zero set of a real polynomial $f_j$ such that $f_j {\mid}_{S_j}$ depends only on variables in $\{x_{j_1}\}_{j_1 \in A_{S_j}}$.
	
\item \label{thm:1.3} For each integer $i \in {\mathbb{N}}_a$, there exists an integer $j_0$ with $A_{i}=A_{S_{j_0}}$ such that $f_{j_0} {\mid}_{S_{j_0}}$ depends on all variables in $\{x_{j_{0,1}}\}_{j_{0,1} \in A_{S_{j_{0}}}}$.
	\item \label{thm:1.4} We have an RA-region $(D_{A_i},\{{\pi}_{n,A_i}(S_{j_3})\}_{j_3 \in \{j_2 \mid A_i=A_{S_{j_2}}\}})$ with $D_{A_i} \subset {\mathbb{R}}^{n_{A_i}}$ being a non-empty open set, $n_{A_i}$ being the size of $A_i$ for each $i \in {\mathbb{N}}_n$, the set $A_i$ being mapped onto $n_{A_i}$ naturally
by the unique map ${\pi}_{A_i,\mathbb{N}}$ preserving the order,
and for "the real polynomial $f_j$ of Definition \ref{def:1}", $f_j {\mid}_{S_j}$ being used instead. 
\item \label{thm:1.5} Distinct sets $A_{i_1}$ and $A_{i_2}$ here have at most one element in common.
\item \label{thm:1.6} If $p_i \in {\overline{D_{A_i}}}^{{\mathbb{R}}^{n_{A_i}}}-D_{A_i}$ is a ${\pi}_{A_i,\mathbb{N}}(N_{A_i,1})$-critical point of $(D_{A_i},\{{\pi}_{n,A_i}(S_{j_3})\}_{j_3 \in \{j_2 \mid A_i=A_{S_{j_2}}\}})$ for a one-element set $N_{A_i,1} \subset A_i$, then for any $i^{\prime} \in {\mathbb{N}}_a$ with $i^{\prime} \neq i$ and $N_{A_i,1} \subset A_{i^{\prime}}$, the set $({\overline{D_{A_{i^{\prime}}}}}^{{\mathbb{R}}^{n_{A_{i^{\prime}}}}}-D_{A_{i^{\prime}}}) \bigcap {\pi}_{n,A_{i^{\prime}}}({{\pi}_{n,A_i}}^{-1}(p_i))$ contains only ${\pi}_{A_{i^{\prime}},\mathbb{N}}(N_{A_i,1})$-N-points of $(D_{A_{i^{\prime}}},\{{\pi}_{n,A_{i^{\prime}}}(S_{j_3})\}_{j_3 \in \{j_2 \mid A_{i^{\prime}}=A_{S_{j_2}}\}})$ and for any $i^{\prime} \in {\mathbb{N}}_a$ with $i^{\prime} \neq i$ and $N_{A_i,1}$ and $A_{i^{\prime}}$ being disjoint, for any one-element set $N_{A_{i^{\prime}},1} \subset A_{i^{\prime}}$, the set $({\overline{D_{A_{i^{\prime}}}}}^{{\mathbb{R}}^{n_{A_{i^{\prime}}}}}-D_{A_{i^{\prime}}}) \bigcap {\pi}_{n,A_{i^{\prime}}}({{\pi}_{n,A_i}}^{-1}(p_i))$ contains only ${\pi}_{A_{i^{\prime}},\mathbb{N}}(N_{A_{i^{\prime}},1})$-N-points of $(D_{A_{i^{\prime}}},\{{\pi}_{n,A_{i^{\prime}}}(S_{j_3})\}_{j_3 \in \{j_2 \mid A_{i^{\prime}}=A_{S_{j_2}}\}})$. 
\item \label{thm:1.7} If $p_i \in {\overline{D_{A_i}}}^{{\mathbb{R}}^{n_{A_i}}}-D_{A_i}$ is not a normal point of $(D_{A_i},\{{\pi}_{n,A_i}(S_{j_3})\}_{j_3 \in \{j_2 \mid A_i=A_{S_{j_2}}\}})$, then for any $i^{\prime} \in {\mathbb{N}}_a$ with $i^{\prime} \neq i$ and any one-element set $N_{A_{i^{\prime}},1} \subset A_{i^{\prime}}$, the set $({\overline{D_{A_{i^{\prime}}}}}^{{\mathbb{R}}^{n_{A_{i^{\prime}}}}}-D_{A_{i^{\prime}}}) \bigcap {\pi}_{n,A_{i^{\prime}}}({{\pi}_{n,A_i}}^{-1}(p_i))$ contains only ${\pi}_{A_{i^{\prime}},\mathbb{N}}(N_{A_{i^{\prime}},1})$-N-points of $(D_{A_i},\{{\pi}_{n,A_i}(S_{j_3})\}_{j_3 \in \{j_2 \mid A_i=A_{S_{j_2}}\}})$. 
\end{enumerate}

Under the present situation, $({\bigcap}_{i=1}^a {{\pi}_{n,A_i}}^{-1}(D_{A_i}),\{S_j\}_{j=1}^l)$ is an RA-region.
\end{Thm}

We prove this, give important examples and present our additional result in the next section. The third section is for additional notes.
\section{On our main result.}

\begin{Def}
	In Theorem \ref{thm:1}, the assumptions (\ref{thm:1.1}), (\ref{thm:1.2}), (\ref{thm:1.3}) and (\ref{thm:1.4}) are called the {\it pre-real-algebraic-region conditions} or the {\it PRAR conditions} for a pair $(\{S_j\}_{j=1}^{l},\{A_i\}_{i=1}^a)$ as presented in the beginning of Theorem \ref{thm:1}, above the assumption (\ref{thm:1.1}). Furthermore, the condition (\ref{thm:1.5}) is called the {\it $b$-intersection condition} for such a pair $(\{S_j\}_{j=1}^{l},\{A_i\}_{i=1}^a)$, where the number $b$ is generalized from $b=1$ to a general integer in ${\mathbb{N}}_{n-1}$.
	\end{Def}
\begin{proof}[A proof of Theorem \ref{thm:1} ]
By the PRAR conditions for $(\{S_j\}_{j=1}^{l},\{A_i\}_{i=1}^a)$, 
we have the condition (\ref{def:1.1}) of Definition \ref{def:1} by defining $U_D:={\bigcap}_{i=1}^a {{\pi}_{n,A_i}}^{-1}(U_{A_i})$ with $U_{A_i}$ being "$U_D$ in the situation of Definition \ref{def:1}".
It is sufficient to prove that the condition (\ref{def:1.2}) of Definition \ref{def:1} is satisfied.

Let $p \in {\bigcap}_{i=1}^a {{\pi}_{n,A_i}}^{-1}(D_{A_i})$. We can have either of the following three cases. \\
\ \\
Case 1-1. For any $i \in {\mathbb{N}}_a$, ${\pi}_{n,A_i}(p)$ is a ${\pi}_{n,A_i}(N_1)$-N-point for any one-element set $N_{A_i,1} \subset A_i$. \\

In this case, for each $A_i$, at most one manifold $S_{i_j}$ of $\{S_j\}_{j=1}^l$ with $A_{S_{i_j}}=A_i$ contains $p$ and the normal vector of $S_{i_j}$ there is of the form ${\Sigma}_{j^{\prime} \in A_i} a_{j^{\prime}}e_{j^{\prime}}$ with all coefficients being non-zero. The condition (\ref{def:1.2}) is satisfied for such a point $p$.\\
\ \\  
Hereafter, for a $k$-element set $N_k \subset {\mathbb{N}}_n$ with $1 \leq k \leq n$, we define ${\mathbb{N}}_{N_k \subset A}:=\{i \in \mathbb{N} \mid N_k \subset A_i\}$. \\
\ \\
Case 1-2.  For some $i_0 \in {\mathbb{N}}_a$, ${\pi}_{n,A_{i_0}}(p)$ is a ${\pi}_{A_{i_0},\mathbb{N}}(N_{A_{i_0},1})$-critical point of $(D_{A_{i_0}},\{{\pi}_{n,A_i}(S_{j_3})\}_{j_3 \in \{j_2 \mid A_{i_0}=A_{S_{j_2}}\}})$ for a one-element set $N_{A_{i_0},1} \subset A_{i_0}$. \\

First, finitely many manifolds $S_{{i_0}_j}$ of $\{S_j\}_{j=1}^l$ with $A_{S_{{i_0}_j}}= A_{i_0}$ contains $p$ and the normal vectors of $S_{{i_0}_j}$ there are of the form ${\Sigma}_{j^{\prime} \in A_{i_0}} a_{j^{\prime}}e_{j^{\prime}}$ with some coefficient being non-zero. They are mutually independent by the definition of an RA-region.

For any element $i_1 \in {\mathbb{N}}_{N_{A_{i_0},1} \subset A}=\{i \in \mathbb{N} \mid N_{A_{i_0},1} \subset A_i\}$, from the condition (\ref{thm:1.6}), at most one manifold $S_{{i_1}_j}$ of $\{S_j\}_{j=1}^l$ with $A_{S_{{i_1}_j}}= A_{i_1}$ contains $p$ and the normal vectors of $S_{{i_1}_j}$ there are of the form of the sum of the vector generated by $\{e_j\}_{j \in N_{A_{i_0},1}}$ and another vector of the form ${\Sigma}_{j^{\prime} \in A_{i_1}-N_{A_{i_0},1}} a_{j^{\prime}}e_{j^{\prime}}$ with some coefficient being non-zero. Remember the $1$-intersection condition for $(\{S_j\}_{j=1}^{l},\{A_i\}_{i=1}^a)$. By this, the sets $A_{i_1}-N_{A_{i_0},1}$ are mutually disjoint for distinct $i_1 \in {\mathbb{N}}_{N_{A_{i_0},1} \subset A}$. 

For any element $i^{\prime} \in {\mathbb{N}}_a-{\mathbb{N}}_{N_{A_{i_0},1} \subset A}$, from the condition (\ref{thm:1.6}), at most one manifold $S_{{i^{\prime}}_j}$ of $\{S_j\}_{j=1}^l$ with $A_{S_{{i^{\prime}}_j}}= A_{i^{\prime}}$ contains $p$ and the normal vectors of $S_{{i^{\prime}}_j}$ there are of the form ${\Sigma}_{j^{\prime} \in A_{i^{\prime}}} a_{j^{\prime}}e_{j^{\prime}}$ with all coefficients being non-zero.

We have explained all normal vectors of all manifolds in the family $\{S_j\}$ at $p$. By the forms of the vectors, all vectors here are mutually independent. \\
\ \\
Case 1-3.  For some $i_0 \in {\mathbb{N}}_a$, ${\pi}_{n,A_{i_0}}(p)$ is not a normal point of $(D_{A_{i_0}},\{{\pi}_{n,A_i}(S_{j_3})\}_{j_3 \in \{j_2 \mid A_{i_0}=A_{S_{j_2}}\}})$. \\

First, finitely many manifolds $S_{{i_0}_j}$ of $\{S_j\}_{j=1}^l$ with $A_{S_{{i_0}_j}}= A_{i_0}$ contains $p$ and the normal vectors of $S_{{i_0}_j}$ there are of the form ${\Sigma}_{j^{\prime} \in A_{i_0}} a_{j^{\prime}}e_{j^{\prime}}$ with some coefficient being non-zero. They are mutually independent by the definition of an RA-region.

For any element $i \in {\mathbb{N}}_a-\{i_0\}$, from the condition (\ref{thm:1.7}), at most one manifold $S_{i_j}$ of $\{S_j\}_{j=1}^l$ with $A_{S_{i_j}}= A_{i}$ contains $p$ and the normal vectors of $S_{i_j}$ there are of the form ${\Sigma}_{j^{\prime} \in A_{i}} a_{j^{\prime}}e_{j^{\prime}}$ with all coefficients being non-zero. 

We have explained all normal vectors of all manifolds in the family $\{S_j\}$ at $p$. By the forms of the vectors, all vectors here are mutually independent. \\

This completes the proof of the fact that the condition (\ref{def:1.2}) of Definition \ref{def:1} is satisfied.

This completes the proof.
\end{proof}
Hereafter, for $1$-dimensional manifolds in ${\mathbb{R}}^n$ ($n \geq 2$), we also use "{\it curves}" and for ($n-1$)-dimensional manifolds in ${\mathbb{R}}^n$ ($n \geq 1$), we also use "{\it hypersurfaces}". 
What follows is a kind of Corollaries to Theorem \ref{thm:1}.

\begin{Thm}
	\label{thm:2}
	For a pair $(\{S_j\}_{j=1}^{l},\{A_i\}_{i=1}^a)$ as presented in the beginning of Theorem \ref{thm:1}, above the assumption {\rm (}\ref{thm:1.1}{\rm )}, let the size of $A_i$ be $1$ or $2$ for each $i$ and suppose that the PRAR conditions for it are satisfied and that the $1$-intersection condition for it is also satisfied. In this situation, the following two are equivalent.
\begin{itemize}
\item The conditions {\rm (\ref{thm:1.6})} and {\rm (\ref{thm:1.7})} in Theorem \ref{thm:1} are satisfied.
\item Let $N_1$ be an arbitrary one-element set of $\mathbb{N}$ contained in at least two distinct sets from $\{A_i\}_{i=1}^a$. For any pair $(A_{{N_1}_1},A_{{N_1}_2})$ of such mutually distinct sets, there does not exist a pair $(p_1,p_2)$ of ${\pi}_{A_{{N_1}_1},\mathbb{N}}(N_1)$-critical point $p_1$ or a point $p_1$ which is not a normal point of $(D_{A_{{N_1}_1}},\{{\pi}_{n,A_{{N_1}_1}}(S_{j_3})\}_{j_3 \in \{j_2 \mid A_{{N_1}_1}=A_{S_{j_2}}\}})$ and a 
${\pi}_{A_{{N_1}_2},\mathbb{N}}(N_1)$-critical point $p_2$ or a point $p_2$ which is not a normal point of
$(D_{A_{{N_1}_2}},\{{\pi}_{n,A_{{N_1}_2}}(S_{j_3})\}_{j_3 \in \{j_2 \mid A_{{N_1}_2}=A_{S_{j_2}}\}})$ such that ${\pi}_{n,N_1}({{\pi}_{n,A_{{N_1}_1}}}^{-1}(p_1))={\pi}_{n,N_1}({{\pi}_{n,A_{{N_1}_2}}}^{-1}(p_2))$, which is a one-point set.
\end{itemize}
	\end{Thm}
We only give related arguments on Theorem \ref{thm:2}.
 
For each point of $\overline{D_{A_i}}^{{\mathbb{R}}^{n_{A_i}}}-D_{A_i}$ which is not a normal point of $(D_{A_i},\{{\pi}_{n,A_i}(S_{j_3})\}_{j_3 \in \{j_2 \mid A_i=A_{S_{j_2}}\}})$, exactly two curves in $\{{\pi}_{n,A_i}(S_{j_3})\}_{j_3 \in \{j_2 \mid A_i=A_{S_{j_2}}\}}$ containing the point exist. The normal vectors of the curves there are mutually independent and span the tangent vector of the plane ${\mathbb{R}}^2$ there.

We can check Theorem \ref{thm:2} from our definitions and formulations. It is a kind of routine works. 

If the latter condition in Theorem \ref{thm:2} is dropped, then we can see that the condition {\rm (\ref{thm:1.6})} or {\rm (\ref{thm:1.7})} in Theorem \ref{thm:1} is dropped by the definition and the observation of $A_{{N_1}_1}$ and $A_{{N_1}_2}$,  
together with the PRAR conditions for $(\{S_j\}_{j=1}^{l},\{A_i\}_{i=1}^a)$.

Conversely, if at least one of the conditions {\rm (\ref{thm:1.6})} and {\rm (\ref{thm:1.7})} in Theorem \ref{thm:1} is dropped, then for some pair $(A_{{N_1}_1},A_{{N_1}_2})$ of mutually distinct sets from $\{A_i\}_{i=1}^a$ and some pair $(p_1,p_2)$ of points as in the condition, we can have ${\pi}_{n,N_1}({{\pi}_{n,A_{{N_1}_1}}}^{-1}(p_1))={\pi}_{n,N_1}({{\pi}_{n,A_{{N_1}_2}}}^{-1}(p_2))$.

We explain several examples for Theorem \ref{thm:1} (Theorem \ref{thm:2}) as Examples \ref{ex:1} and \ref{ex:2}.
\begin{Ex}
\label{ex:1}
Let $n=3$, $a=2$ and $A_j=\{1,j+1\}$. 
Let $S_1=\{(x_1,x_2,x_3) \in {\mathbb{R}}^3 \mid f_1(x_1,x_2,x_3)=1-{(x_1-\frac{1}{2})}^2-{x_2}^2=0\}$ and
$S_2=\{(x_1,x_2,x_3) \in {\mathbb{R}}^3 \mid f_2(x_1,x_2,x_3)=1-{(x_1+\frac{1}{2})}^2-{x_3}^2=0\}$.
Let $A_{S_i}=A_i$.

\end{Ex}
	\begin{Ex}
\label{ex:2}
		In previous studies such as \cite{kitazawa6, kitazawa7, kitazawa16, kitazawa18}, we have encountered cases where $a>1$ $2$-element sets having $\{1\}$ in common and cylinders $S_j$ of curves of degree $2$ in the plane ${\mathbb{R}}^2$ are given.
		For this, see also the third section.
	\end{Ex}
Hereafter, let $N_d \subset {\mathbb{N}}_n$ denote a $d$-element set as before.

In Definition \ref{def:5}, a higher dimensional version of Definition \ref{def:3} is formulated.
\begin{Def}
\label{def:5}

In Definition \ref{def:1}, we define the following.
\begin{enumerate}
\item Let $1 \leq k \leq n-1$ be an integer. We consider the set $S_{i_j,1 \leq j \leq l^{\prime}}$ of all points of ${\overline{D}}^{{\mathbb{R}}^n}-D$ contained in these $l^{\prime}< n-k+1$ distinct hypersurfaces $S_{i_j}$ and not contained in any other hypersurface $S_j$, as in Definition \ref{def:3}. A point $p \in S_{i_j,1 \leq j \leq l^{\prime}}$ is also an {\it $N_k$-singular} point of $(D,\{S_j\}_{j=1}^l)$ if the normal vector of $S_{i_j,1 \leq j \leq l^{\prime}}$ there is of the form ${\Sigma}_{j \in N_k} t_je_j$.
\item If a normal point $p$ of $(D,\{S_j\}_{j=1}^l)$ is not an $N_k$-critical point of $(D,\{S_j\}_{j=1}^l)$ for a $k$-element set $N_k \subset {\mathbb{N}}_n$, then it is called an {\it $N_k$-normal} point or an {\it $N_k$-N-point} of $(D,\{S_j\}_{j=1}^l)$.
\end{enumerate}
\end{Def}
By our formulation, we immediately have the following.
\begin{Prop}

For an RA-region $(D,\{S_j\}_{j=1}^l)$ and a point of ${\overline{D}}^{{\mathbb{R}}^n}$, the following hold.
\begin{enumerate}
\item A point $p \in {\overline{D}}^{{\mathbb{R}}^n}-D$ is an $N_1$-singular point of $(D,\{S_j\}_{j=1}^l)$ if and only if it is an $N_1$-critical point of $(D,\{S_j\}_{j=1}^l)$.

\item We consider $k_1,k_2 \in \mathbb{N}$ and $N_{k_1}$ and $N_{k_2}$ such that the relation $N_{k_1} \subset N_{k_2} \subset {\mathbb{N}}_n$ holds. A point $p \in {\overline{D}}^{{\mathbb{R}}^n}-D$ is an $N_{k_1}$-N-point of $(D,\{S_j\}_{j=1}^l)$ if it is an $N_{k_2}$-N-point of $(D,\{S_j\}_{j=1}^l)$.
\end{enumerate}
\end{Prop}
\begin{Thm}
\label{thm:3}
In Theorem \ref{thm:1}, the PRAR conditions for a pair $(\{S_j\}_{j=1}^{l},\{A_i\}_{i=1}^a)$, the $b$-intersection condition for it, and the condition {\rm (}\ref{thm:1.7}{\rm )} are also satisfied. Furthermore, the following hold.

\begin{enumerate}
	\setcounter{enumi}{7}
\item \label{thm:3.1} If $p_i \in {\overline{D_{A_i}}}^{{\mathbb{R}}^{n_{A_i}}}-D_{A_i}$ is a ${\pi}_{A_i,\mathbb{N}}(N_{A_i,k})$-singular point of $(D_{A_i},\{{\pi}_{n,A_i}(S_{j_3})\}_{j_3 \in \{j_2 \mid A_i=A_{S_{j_2}}\}})$ which is also a normal point of it for a $k$-element set $N_{A_i,k} \subset A_i$ with $1 \leq k \leq b$, then for any $i^{\prime} \in {\mathbb{N}}_a$ with $i^{\prime} \neq i$ and $N_{A_i,k} \subset A_{i^{\prime}}$, the set $({\overline{D_{{A_{i^{\prime}}}}}}^{{\mathbb{R}}^{n_{A_{i^{\prime}}}}}-D_{{A_{i^{\prime}}}}) \bigcap {\pi}_{n,A_{i^{\prime}}}({{\pi}_{n,A_i}}^{-1}(p_i))$ contains only normal points 
 of $(D_{A_{i^{\prime}}},\{{\pi}_{n,A_{i^{\prime}}}(S_{j_3})\}_{j_3 \in \{j_2 \mid A_{i^{\prime}}=A_{S_{j_2}}\}})$, and for any $i^{\prime} \in {\mathbb{N}}_a$ with $i^{\prime} \neq i$ and $N_{A_i,k}-(N_{A_i,k}  \bigcap A_{i^{\prime}})$ being non-empty, for any one-element set $N_{A_{i^{\prime}},1} \subset A_{i^{\prime}}$, the set $({\overline{D_{{A_{i^{\prime}}}}}}^{{\mathbb{R}}^{n_{A_{i^{\prime}}}}}-D_{{A_{i^{\prime}}}}) \bigcap {\pi}_{n,A_{i^{\prime}}}({{\pi}_{n,A_i}}^{-1}(p_i))$ contains only ${\pi}_{A_{i^{\prime}},\mathbb{N}}(N_{A_{i^{\prime}},1})$-N-points of $(D_{A_{i^{\prime}}},\{{\pi}_{n,A_{i^{\prime}}}(S_{j_3})\}_{j_3 \in \{j_2 \mid A_{i^{\prime}}=A_{S_{j_2}}\}})$. 
\item \label{thm:3.2} 
In {\rm (}\ref{thm:3.1}{\rm )}, for any family $\{p_{i^{\prime \prime}}\}$ of finitely many normal points of $(D_{A_{i^{\prime \prime}}},\{{\pi}_{n,A_{i^{\prime \prime}}}(S_{j_3})\}_{j_3 \in \{j_2 \mid A_{i^{\prime \prime}}=A_{S_{j_2}}\}})$ chosen for mutually distinct $i^{\prime \prime} \in {\mathbb{N}}_a$ with $N_{A_i,k} \subset A_{i^{\prime \prime}}$, the normal vectors of the unique hypersurfaces in $\{{\pi}_{n,A_{i^{\prime \prime}}}(S_{j_3})\}_{j_3 \in \{j_2 \mid A_{i^{\prime \prime}}=A_{S_{j_2}}\}}$ at the points $p_{i^{\prime \prime}}$ are mapped to mutually independent vectors by the differentials of the projections ${\pi}_{n_{A_{i^{\prime \prime}}},{\pi}_{{A_i}^{\prime \prime},\mathbb{N}}(N_{A_i,k})}$ if the points are all mapped to a same point ${\pi}_{n_{A_{i^{\prime \prime}}},{\pi}_{{A_i}^{\prime \prime},\mathbb{N}}(N_{A_i,k})}(p_{i^{\prime \prime}})$ by ${\pi}_{n_{A_{i^{\prime \prime}}},{\pi}_{{A_i}^{\prime \prime},\mathbb{N}}(N_{A_i,k})}$.
\end{enumerate}
Under the present situation, $({\bigcap}_{i=1}^a {{\pi}_{n,A_i}}^{-1}(D_{A_i}),\{S_j\}_{j=1}^l)$ is an RA-region.
\end{Thm}
\begin{proof}
	This is regarded as an extension of Theorem \ref{thm:1} to the $b$-intersection condition with $1 \leq b \leq n-1$.
	
	For the proof, it is sufficient to consider a generalization of Case 1-2 and we discuss the generalized case.
	
    First, we remember the condition (\ref{thm:3.1}) here.
	 
	It is sufficient to consider the case, where for an integer $i_0 \in {\mathbb{N}}_a$, ${\pi}_{n,A_{i_0}}(p)$ is a ${\pi}_{A_{i_0},\mathbb{N}}(N_{A_{i_0},k})$-singular point of $(D_{A_{i_0}},\{{\pi}_{n,A_i}(S_{j_3})\}_{j_3 \in \{j_2 \mid A_{i_0}=A_{S_{j_2}}\}})$ and a normal point of it for the $k$-element set $N_{A_{i_0},k} \subset A_{i_0}$ with $1 \leq k \leq b$. 
	
	Only one manifold $S_{{i_0}_j}$ of $\{S_j\}_{j=1}^l$ with $A_{S_{{i_0}_j}}= A_{i_0}$ contains $p$. The normal vector of $S_{{i_0}_j}$ there is of the form ${\Sigma}_{j^{\prime} \in A_{i_0}} a_{j^{\prime}}e_{j^{\prime}}$ with some coefficient being non-zero. For $i_1 \in {\mathbb{N}}_{N_{A_{i_0},k} \subset A}=\{i \in \mathbb{N} \mid N_{A_{i_0},k} \subset A_i\}$, at most one manifold $S_{{i_1}_j}$ of $\{S_j\}_{j=1}^l$ with $A_{S_{{i_1}_j}}=A_{i_1}$ contains $p$ and the normal vectors of $S_{{i_1}_j}$ there are of the form of the sum of the vector generated by $\{e_j\}_{j \in N_{A_{i_0},k}}$ and another vector of the form ${\Sigma}_{j^{\prime} \in A_{i_1}-N_{A_{i_0},k}} a_{j^{\prime}}e_{j^{\prime}}$.
	Remember the condition (\ref{thm:3.2}), where $\{p_{i^{\prime \prime}}\}$ is labeled by an arbitrary finite subset of ${\mathbb{N}}_{N_{A_{i_0},k} \subset A}$. The presented normal vectors of the manifolds are mutually independent.

	
	For any element $i^{\prime} \in {\mathbb{N}}_a-{\mathbb{N}}_{N_{A_{i_0},k} \subset A}$, from the condition (\ref{thm:3.1}), at most one manifold $S_{{i^{\prime}}_j}$ of $\{S_j\}_{j=1}^l$ with $A_{S_{{i^{\prime}}_j}}= A_{i^{\prime}}$ contains $p$ and the normal vectors of $S_{{i^{\prime}}_j}$ there are of the form ${\Sigma}_{j^{\prime} \in A_{i^{\prime}}} a_{j^{\prime}}e_{j^{\prime}}$ with all coefficients being non-zero.
	
Remember that $A_{i_1} \subset A_{i_2}$ must not hold for any distinct integers $i_1$ and $i_2$, again. 	
	
	We have explained all normal vectors of all manifolds in the family $\{S_j\}$ at $p$. By the forms of the vectors, all vectors here are mutually independent.
	
	This completes the proof.
	\end{proof}

In Example \ref{ex:3}, an explicit case for Theorem \ref{thm:3} is presented.
\begin{Ex}
\label{ex:3}
Let $n=4$, $a=2$, $b=2$, and $A_j=\{1,2,j+2\}$. 
We choose a sufficiently small positive number $r>0$.
Let $S_1=\{(x_1,x_2,x_3,x_4) \in {\mathbb{R}}^4 \mid f_1(x_1,x_2,x_3,x_4)=1-{(x_1-\frac{1}{2})}^2-{x_2}^2-{x_3}^2=0\}$ and 
$S_2=\{(x_1,x_2,x_3,x_4) \in {\mathbb{R}}^4 \mid f_2(x_1,x_2,x_3,x_4)=1+r-{(x_1+\frac{1}{2})}^2-{x_2}^2-{x_4}^2=0\}$.

Let $N_{A_1,2}=\{1,2\} \subset A_1, A_2$.
The sets $S_{1,x_3=x_4=0}=\{(x_1,x_2,0,0) \in {\mathbb{R}}^4 \mid f_1(x_1,x_2,x_3,x_4)=1-{(x_1-\frac{1}{2})}^2-{x_2}^2=0\}$ and 
$S_{2,x_3=x_4=0}=\{(x_1,x_2,0,0) \in {\mathbb{R}}^4 \mid f_2(x_1,x_2,x_3,x_4)=1+r-{(x_1+\frac{1}{2})}^2-{x_2}^2=0\}$ are regarded as the set of all ${\pi}_{A_1,\mathbb{N}}(N_{A_1,2})$-singular points of $(D_{A_1},\{{\pi}_{4,A_1}(S_{j_3})\}_{j_3 \in \{j_2 \mid A_1=A_{S_{j_2}}\}})$ and that of all ${\pi}_{A_2,\mathbb{N}}(N_{A_1,2})$-singular points of $(D_{A_2},\{{\pi}_{4,A_2}(S_{j_3})\}_{j_3 \in \{j_2 \mid A_2=A_{S_{j_2}}\}})$, respectively.

Let $A_{S_i}=A_i$.

\end{Ex}
\section{Additional notes.}
As presented in preprints by the author in \cite{kitazawa9, kitazawa10, kitazawa11, kitazawa12, kitazawa13, kitazawa14, kitazawa15, kitazawa16, kitazawa18} for example, we explain a real algebraic map locally like a moment map onto the closure ${\overline{D}}^{{\mathbb{R}}^n}$ in Definition \ref{def:1}. This is originally presented as \cite[Main Theorem 1]{kitazawa4}.
\begin{Thm}[\cite{kitazawa4}]
\label{thm:4}
	In Definition \ref{def:1}, let $l_{\rm S}:{\mathbb{N}}_{l} \rightarrow {\mathbb{N}}_{l_a}$ be a surjection with $l_a \in \mathbb{N}$ such that $l_{\rm S}(i_1) \neq l_{\rm S}(i_2)$ for distinct two elements $i_1,i_2 \in {\mathbb{N}}_{l}$ with $S_{i_1} \bigcap S_{i_2} \bigcap {\overline{D}}^{{\mathbb{R}}^n} \neq \emptyset$ and let $l_{\rm S,\mathbb{N}}:{\mathbb{N}}_{l_a} \rightarrow \mathbb{N} \sqcup \{0\}$ be a function. Let the value of  $l_{\rm S,\mathbb{N}}$ at each $i \in {\mathbb{N}}_{l_a}$ be denoted by $d_{i,l_a}$.
Then we have a real algebraic manifold 
$M_{(D,\{S_j\}_{j=1}^l,\{d_{i,l_a}\}_{i=1}^{l_a})}:={\bigcap}_{i=1}^{l_a} \{(x,y_1,\cdots,y_{l_a}) \mid y_i=(y_{i,j^{i,l_q}})_{j^{i,l_a}=1}^{d_{i,l_a}+1}, {\prod}_{j_i \in {l_{\rm S}}^{-1}(i)} (f_{j_i}(x))-{\Sigma}_{j^{i,l_a}=1}^{d_{i,l_a}+1} {y_{i,j^{i,l_a}}}^2=0\}$ and a real algebraic map defined as the restriction ${\pi}_{n+{\Sigma}_{i=1}^{l_a} (d_{i,l_a}+1),n} {\mid}_{M_{(D,\{S_j\}_{j=1}^l,\{d_{i,l_a}\}_{i=1}^{l_a})}}$.
\end{Thm}
We are interested in topological properties and combinatorial ones of the function ${\pi}_{n+{\Sigma}_{i=1}^{l_a} (d_{i,l_a}+1),1} {\mid}_{M_{(D,\{S_j\}_{j=1}^l,\{d_{i,l_a}\}_{i=1}^{l_a})}}$.
Such properties are understood by the {\it Reeb {\rm (}di{\rm )}graphs} of the function.

For a smooth function on a manifold $X$ with no boundary $c:X \rightarrow \mathbb{R}$, we can define the quotient space $R_c$ and the quotient map $q_{c}:X \rightarrow R_c$ and can have the unique continuous function $\bar{c} \rightarrow \mathbb{R}$ with $c=\bar{c} \circ q_c$.
In the case $X$ is a closed manifold with the image of the set of all singular points of $c$ being finite and explicit non-proper cases such as ones presented in \cite{kitazawa2}, for example, this quotient space $R_c$ has the structure of a graph as follows. According to \cite[Theorem 3.1]{saeki2} and \cite[Theorem 1]{saeki3}, $R_c$ is a graph where a point $v$ with ${q_c}^{-1}(v)$ containing some singular points of $c$ is defined to be its vertex in the former case. The latter case can be understood by observing local behavior of the functions. This
 graph $R_c$ is the {\it Reeb graph} of $c$. This is also a digraph where each edge $e$ is oriented from a vertex $v_{e,1}$ to another vertex $v_{e,2}$ where the relation $\bar{c}(v_{e,1})<\bar{c}(v_{e,2})$ and the {\it Reeb digraph} $\overrightarrow{R_c}$ of $c$ is obtained.

We go back to Theorem \ref{thm:4}. This is motivated by the real algebraic version of reconstruction of nice smooth functions with prescribed Reeb {\rm (}di{\rm )}graphs, started by Sharko in \cite{sharko}, followed by studies such as \cite{masumotosaeki, michalak} and the author also follows them by respecting shapes of level sets, not only the Reeb graphs in \cite{kitazawa1, kitazawa2}. The paper \cite{kitazawa3} is a related pioneering paper on real algebraic construction, followed first by \cite{kitazawa4, kitazawa5, kitazawa6, kitazawa7} and remarked by \cite{kitazawa8}, by the author himself. In our present studies and closely related ones, we are trying to contribute to systematic studies on shapes of real algebraic manifolds and functions and maps on them.  
\section{Conflict of interest and data availability.}
\noindent {\bf Conflict of interest.} \\
The author is also a researcher at Osaka Central
Advanced Mathematical Institute (OCAMI researcher), supported by MEXT Promotion of Distinctive Joint Research Center Program JPMXP0723833165. He is not employed there. This is for our studies
and our study also thanks this. \\
\ \\
{\bf Data availability.} \\
No other data is associated to the paper.

\end{document}